\documentclass{amsart}

\usepackage{ amsthm }
\usepackage{ amssymb }
\usepackage{ amsmath }
\usepackage{ array }
\usepackage{ booktabs }
\usepackage{ cite }
\usepackage{ color }
\usepackage{ enumitem }
\usepackage{ latexsym }
\usepackage{ url }

\theoremstyle{definition}

\newtheorem{definition}{Definition}[section]
\newtheorem{algorithm}[definition]{Algorithm}
\newtheorem{example}[definition]{Example}
\newtheorem{remark}[definition]{Remark}

\theoremstyle{plain}

\newtheorem{lemma}[definition]{Lemma}
\newtheorem{theorem}[definition]{Theorem}

\newcommand{\free}{\mathbf{Free}} 
\newcommand{\dias}{\mathbf{Dias}} 

\allowdisplaybreaks

\begin{document}

\title{SYMMETRIZATION OF JORDAN DIALGEBRAS}

\author{Murray R. Bremner}

\address{Department of Mathematics and Statistics, University of Saskatchewan, Canada}

\email{bremner@math.usask.ca}





\thanks{The author's research was supported by the Discovery Grant \emph{Algebraic Operads} from NSERC, 
the Natural Sciences and Engineering Research Council of Canada.
He thanks Vladimir Dotsenko for introducing him to the problem of polynomial identities satisfied by 
bilinear operations in diassociative algebras.
This work was completed in November 2017 during the author's visit to Suleyman Demirel University 
in Kaskelen, Kazakhstan.
He thanks Askar Dzhumadildaev and Nurlan Ismailov, and their colleagues and students, 
for providing a stimulating research environment.}

\begin{abstract}
A basic problem for any class of nonassociative algebras is to determine the polynomial identities
satisfied by the symmetrization and the skew-symmetrization of the original product.
We consider the symmetrization of the product in the class of special Jordan dialgebras.
We use computational linear algebra to show 
that every polynomial identity of degree $n \le 5$ satisfied by the symmetrized Jordan diproduct in 
every diassociative algebra is a consequence of commutativity.
We determine a complete set of generators for the polynomial identities in degree 6 which are not
consequences of commutativity.
We use a constructive version of the representation theory of the symmetric group to show that there 
exist further new identities in degree 7.
\end{abstract}

\maketitle

\thispagestyle{empty}


\section{Introduction}

Let $\mathcal{V}$ be a class of nonassociative algebras%
\footnote{We interpret ``nonassociative'' loosely to mean ``not necessarily associative''.} 
over a field with a single bilinear operation denoted $a \cdot b$ which is neither commutative
nor anticommutative.
For background on the theory of varieties of nonassociative algebras defined by 
polyomial identities, we refer to Osborn's long paper \cite{Osborn}, the ``Russian book'' \cite{ZSSS}, 
and the monographs on Jordan algebras by Jacobson \cite{JacobsonJ} and McCrimmon \cite{McCrimmon}.

For an algebra $A \in \mathcal{V}$ we define two other operations on the underlying vector space, 
one commutative and one anticommutative:
the symmetrized product (anticommutator, Jordan product) $\{a,b\} = a \cdot b + b \cdot a$,
and the skew-symmetrized product (commutator, Lie bracket)
$[a,b] = a \cdot b - b \cdot a$.
The resulting algebras are denoted $A^+$ and $A^-$, and called the plus and minus algebras of $A$.
This process is closely related to the polarization of operations studied by Markl \& Remm \cite{MarklRemm}.

For a given class $\mathcal{V}$, a basic problem in nonassociative algebra is to determine 
the polynomial identities satisfied by all algebras of the form $A^+$, respectively $A^-$.
In the most familiar case, the variety of associative algebras, it is known that:
\begin{itemize}[leftmargin=*]
\item[$\bullet$]
Every algebra $A^+$ satisfies the Jordan identity, and every polynomial identity of degree $n \le 7$ 
satisfied by every $A^+$ follows from commutativity and the Jordan identity.
However, there are ``special'' identities of degree $n \ge 8$ which are satisfied by every $A^+$ 
but which are not consequences of commutativity and the Jordan identity.
The simplest of these were discovered by Glennie \cite{Glennie1966}.
\item[$\bullet$]
Every algebra $A^-$ satisfies the Jacobi identity, and the Poincar\'e-Birkhoff-Witt theorem 
implies that every polynomial identity of every degree satisfied by every $A^+$ is a consequence of 
anticommutativity and the Jacobi identity.
\end{itemize}

\begin{definition}
\label{dadef}
Loday \cite{L1995,L2001} introduced the notion of \emph{diassociative algebra} 
(or associative dialgebra), which is a vector space with two bilinear operations $\vdash$ and $\dashv$, 
the \emph{right and left products}, satisfying the following polynomial identities:
\[
\begin{array}{l@{\qquad}l}
( x \vdash y ) \vdash z \equiv x \vdash ( y \vdash z ), &
( x \dashv y ) \dashv z \equiv x \dashv ( y \dashv z ),
\\
( x \vdash y ) \dashv z \equiv x \vdash ( y \dashv z ),
\\
( x \dashv y ) \vdash z \equiv ( x \vdash y ) \vdash z, &
x \dashv ( y \dashv z ) \equiv x \dashv ( y \vdash z ).
\end{array}
\]
These are \emph{right, left}, and \emph{inner associativity}, and the \emph{left} and \emph{right bar identities}.
\end{definition}

In the setting of diassociative algebras, the analogues of the Jordan product and the Lie bracket are 
the Jordan diproduct (antidicommutator) 
$\{a,b\} = \widehat{a} \, b + b \, \widehat{a}$
and the Leibniz bracket (dicommutator)
$[a,b] = \widehat{a} \, b - b \, \widehat{a}$.
It is known that:
\begin{itemize}[leftmargin=*]
\item[$\bullet$]
The Jordan diproduct in every diassociative algebra satisfies polynomial identities in degrees 3 and 4
which define the variety of Jordan dialgebras: 
see Kolesnikov \cite{K}, Vel\'asquez \& Felipe \cite{VF}, Bremner \cite{B}.
Every identity of degree $n \le 7$ follows from these identities.
However, there are ``special'' identities in degree 8 which are not consequences of these identities:
the simplest were discovered by Bremner \& Peresi \cite{BP2011}; see also Voronin \cite{V},
Kolesnikov \& Voronin \cite{KV}.
\item[$\bullet$]
The Leibniz bracket satisfies the derivation identity in degree 3 which defines the variety of Leibniz
algebras; see Loday \cite{L1993}.
A generalization of the Poincar\'e-Birkhoff-Witt theorem implies that every polynomial 
identity of every degree satisfied by the Leibniz bracket in every diassociative algebra is a consequence of 
this identity; see Loday \& Pirashvili \cite{LP}, Aymon \& Grivel \cite{AG}, Bokut et al.~\cite{BCL}.
\end{itemize}
Since the Jordan diproduct and the Leibniz bracket are neither commutative nor anticommutative,
it is a basic problem to determine the polynomial identities satisfied by the algebras $A^+$ and $A^-$ 
where $A$ is a Jordan dialgebra or a Leibniz algebra.
This paper studies the algebras $A^+$ where $A$ is a Jordan dialgebra.

\begin{definition}
\label{jddef}
A (left) \emph{Jordan dialgebra} is a vector space with a bilinear operation denoted $(x,y) \mapsto xy$ 
satisfying the following polynomial identities: 
\[
x \cdot (y \cdot z) \equiv x \cdot (z \cdot y),
\qquad
(y \cdot x) \cdot x^2 \equiv (y \cdot x^2) \cdot x,
\qquad
(y,x^2,z) \equiv 2(y,x,z) \cdot z,
\]
where $x^2 = x \cdot x$ and $(x,y,z) = (x \cdot y) \cdot z - x \cdot (y \cdot z)$ is the associator.
\end{definition}

We are concerned with the polynomial identities satisfied by the symmetrization of 
the Jordan diproduct in every diassociative algebra:
\[
xy = x \cdot y + y \cdot x = \widehat{x} \, y + \widehat{y} \, x + y \, \widehat{x} + x \, \widehat{y}.
\]
We work over a field of characteristic 0, unless otherwise indicated.
This assumption implies that every polynomial identity is equivalent to a finite set of multilinear identities
\cite[Chapter 1]{ZSSS}, so we may apply the representation theory of the symmetric group \cite{BMP}.
Some large computations require arithmetic modulo a prime $p$ to reduce memory usage.
Since the structure constants for the group algebra $\mathbb{Q} S_n$ have denominators which are divisors 
of $n!$, if we use a prime $p > n$ where $n$ is the degree of the identities, then the group algebra 
$\mathbb{F}_p S_n$ is semisimple and we can apply rational reconstruction to recover the results in 
characteristic 0.
We do not distinguish between the polynomial identity $f \equiv g$ and the polynomial $f - g$.


\section{Algebraic Operads}

The results of this paper may be conveniently formulated in the language of algebraic operads 
\cite{BD,LV,MSS}.
The operads we consider are symmetric operads in the symmetric monoidal category of vector spaces 
over a field of characteristic 0; the product is the tensor product, and the coproduct 
is the direct sum.

\begin{definition}
We write $\free$ for the free (symmetric) operad generated by a commutative (nonassociative) 
binary operation $\omega$.
We write $\textbf{Dias}$ for the (symmetrization of the nonsymmetric) diassociative operad
generated by the right and left operations $\rho$ and $\lambda$.
A morphism $X\colon \free \longrightarrow \dias$ is uniquely determined by its action on $\omega$; we therefore define the \textit{expansion map} by
\[
X( \omega ) = \lambda + \lambda^{(12)} + \rho + \rho^{(12)}.
\]
The right side of this equation is the operadic form of the symmetrized Jordan diproduct
where the superscript permutations act on the arguments.
\end{definition}

\begin{lemma}
The following dimension formulas are well-known:
\[
\dim \free(n) = (2n{-}3)!! \;\; \text{where} \;\;
n!! = \prod_{i=0}^{\lfloor n/2 \rfloor} (n{-}2i),
\qquad
\dim \dias(n) = n(n!).
\]
\end{lemma}

\begin{remark}
The number of inequivalent association types (placements of parentheses) for a 
commutative nonassociative operation is given by the sequence of Wedderburn-Etherington numbers \cite{BF},
which begins 1, 1, 1, 2, 3, 6, 11, \dots.
\end{remark}

\begin{algorithm}
Loday \cite{L1995,L2001} proved that the diassociative identities imply 
a simple normal form for diassociative monomials $m = \overline{x_1 \cdots x_n}$ 
where the overline indicates an arbitrary placement of parentheses and an arbitrary assignment 
of right and left operation symbols.
We express $m$ as a plane rooted complete binary tree $t$ with $n$ leaves labelled $x_1, \dots, x_n$ from 
left to right and $n{-}1$ internal nodes (including the root) labelled by operation symbols.
Starting at the root, we follow the path determined by the operations:
$\vdash$ or $\dashv$ indicate respectively that we choose the right or left subtree.
This path terminates at a unique leaf $x_i$, called the center (or middle) of the diassociative monomial $m$.
It follows that
\[
m = x_1 \vdash \cdots \vdash x_{i-1} \vdash x_i \dashv x_{i+1} \dashv \cdots \dashv x_n,
\]
and that this expression is independent of the placement of parentheses.
This allows us to omit the operation symbols and denote the center by a hat:
\[
m = x_1 \cdots x_{i-1} \, \widehat{x}_i \, x_{i+1} \cdots x_n.
\]
Multiplication of monomials then takes the following simple form where
the direction of the operation symbol determines the center of the product:
\[
\begin{array}{l}
x_1 \cdots \widehat{x}_i \cdots x_p
\,\vdash\,
y_1 \cdots \widehat{y}_j \cdots y_q
=
x_1 \cdots x_i \cdots x_p y_1 \cdots \widehat{y}_j \cdots y_q,
\\
x_1 \cdots \widehat{x}_i \cdots x_p
\,\dashv\,
y_1 \cdots \widehat{y}_j \cdots y_q
=
x_1 \cdots \widehat{x}_i \cdots x_p y_1 \cdots y_j \cdots y_q;
\end{array}
\]
\end{algorithm}

\begin{example}
\label{ex3}
In degree 3, we have the following ordered monomial bases:
\[
\free(3)\colon \; (ab)c, \; (ac)b, \; (bc)a;
\qquad
\dias(3)\colon \;
\widehat{a^\sigma} b^\sigma c^\sigma, \; 
a^\sigma \widehat{b^\sigma} c^\sigma, \;
a^\sigma b^\sigma \widehat{c^\sigma} \; ( \sigma \in S_3 ).
\]
(The elements of $S_3$ are in lex order.)
The following formula for $X( (ab)c )$ is easily verified, 
and permutation of the arguments gives $X( (ac)b )$ and $X( (bc)a )$:
\[
    \widehat{a} b c 
+   \widehat{b} a c 
+ 2 \widehat{c} a b 
+ 2 \widehat{c} b a 
+   a \widehat{b} c 
+   b \widehat{a} c 
+   c \widehat{a} b 
+   c \widehat{b} a 
+ 2 a b \widehat{c} 
+ 2 b a \widehat{c}
+   c a \widehat{b} 
+   c b \widehat{a}. 
\]
We obtain (the transpose of) the matrix representing $X$ in degree 3 (dot for zero):
\[
\left[
\begin{array}{cccccccccccccccccc}
1 & . & 1 & . & 2 & 2 & 1 & . & 1 & . & 1 & 1 & 2 & . & 2 & . & 1 & 1 \\[-2pt]
. & 1 & 2 & 2 & 1 & . & . & 1 & 1 & 1 & 1 & . & . & 2 & 1 & 1 & 2 & . \\[-2pt]
2 & 2 & . & 1 & . & 1 & 1 & 1 & . & 1 & . & 1 & 1 & 1 & . & 2 & . & 2
\end{array}
\right]
\]
This matrix has rank 3, so every identity in degree 3 follows from commutativity.
\end{example}


\section{Degrees 4 and 5}

\begin{lemma}
\label{lemma4}
Every multilinear polynomial identity of degree $n \le 4$ satisfied by the symmetrization of 
the Jordan diproduct is a consequence of commutativity.
\end{lemma}

\begin{figure}[ht]
$\left[
\begin{array}{ccccccccccccccc}
1 & \cdot & \cdot & \cdot & \cdot & \cdot & 2 & 4 & \cdot & \cdot & \cdot & 4 & 2 & \cdot & \cdot \\[-2pt]
\cdot & 1 & \cdot & \cdot & \cdot & \cdot & \cdot & \cdot & 2 & 4 & \cdot & 4 & 2 & \cdot & \cdot \\[-2pt]
\cdot & \cdot & 1 & \cdot & \cdot & \cdot & 2 & 4 & \cdot & 4 & \cdot & \cdot & \cdot & 2 & \cdot \\[-2pt]
\cdot & \cdot & \cdot & 1 & \cdot & \cdot & \cdot & \cdot & \cdot & 4 & 2 & 4 & \cdot & 2 & \cdot \\[-2pt]
\cdot & \cdot & \cdot & \cdot & 1 & \cdot & \cdot & 4 & 2 & 4 & \cdot & \cdot & \cdot & \cdot & 2 \\[-2pt]
\cdot & \cdot & \cdot & \cdot & \cdot & 1 & \cdot & 4 & \cdot & \cdot & 2 & 4 & \cdot & \cdot & 2 \\[-2pt]
1 & \cdot & 2 & 4 & \cdot & \cdot & \cdot & \cdot & \cdot & \cdot & 4 & \cdot & 2 & \cdot & \cdot \\[-2pt]
\cdot & 1 & \cdot & \cdot & 2 & 4 & \cdot & \cdot & \cdot & \cdot & 4 & \cdot & 2 & \cdot & \cdot \\[-2pt]
\cdot & \cdot & 2 & 4 & \cdot & 4 & 1 & \cdot & \cdot & \cdot & \cdot & \cdot & \cdot & \cdot & 2 \\[-2pt]
\cdot & \cdot & \cdot & \cdot & \cdot & 4 & \cdot & 1 & \cdot & \cdot & 4 & 2 & \cdot & \cdot & 2 \\[-2pt]
\cdot & \cdot & \cdot & 4 & 2 & 4 & \cdot & \cdot & 1 & \cdot & \cdot & \cdot & \cdot & 2 & \cdot \\[-2pt]
\cdot & \cdot & \cdot & 4 & \cdot & \cdot & \cdot & \cdot & \cdot & 1 & 4 & 2 & \cdot & 2 & \cdot \\[-2pt]
2 & 4 & 1 & \cdot & \cdot & \cdot & \cdot & \cdot & 4 & \cdot & \cdot & \cdot & \cdot & 2 & \cdot \\[-2pt]
\cdot & \cdot & \cdot & 1 & 4 & 2 & \cdot & \cdot & 4 & \cdot & \cdot & \cdot & \cdot & 2 & \cdot \\[-2pt]
1 & \cdot & \cdot & \cdot & \cdot & \cdot & 1 & 1 & \cdot & \cdot & \cdot & 2 & 2 & \cdot & \cdot
\end{array}
\right]$
\caption{Full rank submatrix for proof of Lemma \ref{lemma4}}
\label{submatrix4}
\end{figure}

\begin{proof}
Similar to Example \ref{ex3}, but the matrix is larger.
In degree 4, the two association types $((--)-)-$ and $(--)(--)$ for a commutative nonassociative 
operation have respectively 12 and 3 multilinear monomials, ordered by type and then by permutation 
of the arguments.
We expand the monomials with the identity permutation into the diassociative operad 
using the symmetrized Jordan diproduct:
\[
\begin{array}{r@{\;}l}
X(((ab)c)d)
&=
    \widehat{a} b c d   
+   \widehat{b} a c d   
+ 2 \widehat{c} a b d   
+ 2 \widehat{c} b a d   
+ 4 \widehat{d} a b c   
+ 4 \widehat{d} b a c   
+ 4 \widehat{d} c a b   
+ 4 \widehat{d} c b a   
\\
&{}
+   a \widehat{b} c d   
+   b \widehat{a} c d   
+   c \widehat{a} b d   
+   c \widehat{b} a d   
+   d \widehat{a} b c   
+   d \widehat{b} a c   
+ 2 d \widehat{c} a b   
+ 2 d \widehat{c} b a   
\\
&{}
+ 2 a b \widehat{c} d   
+ 2 b a \widehat{c} d   
+   c a \widehat{b} d   
+   c b \widehat{a} d   
+   d a \widehat{b} c   
+   d b \widehat{a} c   
+   d c \widehat{a} b   
+   d c \widehat{b} a   
\\
&{}
+ 4 a b c \widehat{d}   
+ 4 b a c \widehat{d}   
+ 4 c a b \widehat{d}   
+ 4 c b a \widehat{d}   
+ 2 d a b \widehat{c}   
+ 2 d b a \widehat{c}   
+   d c a \widehat{b}   
+   d c b \widehat{a},  
\\[3pt]
X((ab)(cd))
&=
  2 \widehat{a} b c d   
+ 2 \widehat{a} b d c   
+ 2 \widehat{b} a c d   
+ 2 \widehat{b} a d c   
+ 2 \widehat{c} d a b   
+ 2 \widehat{c} d b a   
+ 2 \widehat{d} c a b   
+ 2 \widehat{d} c b a   
\\
&{}
+ 2 a \widehat{b} c d   
+ 2 a \widehat{b} d c   
+ 2 b \widehat{a} c d   
+ 2 b \widehat{a} d c   
+ 2 c \widehat{d} a b   
+ 2 c \widehat{d} b a   
+ 2 d \widehat{c} a b   
+ 2 d \widehat{c} b a   
\\
&{}
+ 2 a b \widehat{c} d   
+ 2 a b \widehat{d} c   
+ 2 b a \widehat{c} d   
+ 2 b a \widehat{d} c   
+ 2 c d \widehat{a} b   
+ 2 c d \widehat{b} a   
+ 2 d c \widehat{a} b   
+ 2 d c \widehat{b} a   
\\
&{}
+ 2 a b c \widehat{d}   
+ 2 a b d \widehat{c}   
+ 2 b a c \widehat{d}   
+ 2 b a d \widehat{c}   
+ 2 c d a \widehat{b}   
+ 2 c d b \widehat{a}   
+ 2 d c a \widehat{b}   
+ 2 d c b \widehat{a}.  
\end{array}
\]
Permutation of the arguments gives the expansions of the other monomials.
In degree 4, we order the $4 \cdot 4! = 96$ multilinear diassociative monomials by position of the center 
and then by permutation of the arguments.
We construct the $96 \times 15$ matrix $E$ whose $(i,j)$ entry is the coefficient of 
the $i$-th diassociative monomial in the expansion of the $j$-th commutative nonassociative monomial.
The submatrix consisting of rows 1--14 and 25 corresponds to the diassociative monomials
$\widehat{a}bcd$, \dots, $\widehat{c}adb, a\widehat{b}cd$ (Figure \ref{submatrix4}).
This submatrix has full rank which completes the proof.
\end{proof}

\begin{figure}[ht]
$\begin{array}{l}
X( (((ab)c)d)e ) =
\\
\;\;\;\;\,
    \widehat{a} b c d e   
+   \widehat{b} a c d e   
+ 2 \widehat{c} a b d e   
+ 2 \widehat{c} b a d e   
+ 4 \widehat{d} a b c e   
+ 4 \widehat{d} b a c e   
+ 4 \widehat{d} c a b e   
+ 4 \widehat{d} c b a e   
\\
{} 
+ 8 \widehat{e} a b c d   
+ 8 \widehat{e} b a c d   
+ 8 \widehat{e} c a b d   
+ 8 \widehat{e} c b a d   
+ 8 \widehat{e} d a b c   
+ 8 \widehat{e} d b a c   
+ 8 \widehat{e} d c a b   
+ 8 \widehat{e} d c b a   
\\
{} 
+   a \widehat{b} c d e   
+   b \widehat{a} c d e   
+   c \widehat{a} b d e   
+   c \widehat{b} a d e   
+   d \widehat{a} b c e   
+   d \widehat{b} a c e   
+ 2 d \widehat{c} a b e   
+ 2 d \widehat{c} b a e   
\\
{} 
+   e \widehat{a} b c d   
+   e \widehat{b} a c d   
+ 2 e \widehat{c} a b d   
+ 2 e \widehat{c} b a d   
+ 4 e \widehat{d} a b c   
+ 4 e \widehat{d} b a c   
+ 4 e \widehat{d} c a b   
+ 4 e \widehat{d} c b a   
\\
{} 
+ 2 a b \widehat{c} d e   
+ 2 b a \widehat{c} d e   
+   c a \widehat{b} d e   
+   c b \widehat{a} d e   
+   d a \widehat{b} c e   
+   d b \widehat{a} c e   
+   d c \widehat{a} b e   
+   d c \widehat{b} a e   
\\
{} 
+   e a \widehat{b} c d   
+   e b \widehat{a} c d   
+   e c \widehat{a} b d   
+   e c \widehat{b} a d   
+   e d \widehat{a} b c   
+   e d \widehat{b} a c   
+ 2 e d \widehat{c} a b   
+ 2 e d \widehat{c} b a   
\\
{} 
+ 4 a b c \widehat{d} e   
+ 4 b a c \widehat{d} e   
+ 4 c a b \widehat{d} e   
+ 4 c b a \widehat{d} e   
+ 2 d a b \widehat{c} e   
+ 2 d b a \widehat{c} e   
+   d c a \widehat{b} e   
+   d c b \widehat{a} e   
\\
{} 
+ 2 e a b \widehat{c} d   
+ 2 e b a \widehat{c} d   
+   e c a \widehat{b} d   
+   e c b \widehat{a} d   
+   e d a \widehat{b} c   
+   e d b \widehat{a} c   
+   e d c \widehat{a} b   
+   e d c \widehat{b} a   
\\
{} 
+ 8 a b c d \widehat{e}   
+ 8 b a c d \widehat{e}   
+ 8 c a b d \widehat{e}   
+ 8 c b a d \widehat{e}   
+ 8 d a b c \widehat{e}   
+ 8 d b a c \widehat{e}   
+ 8 d c a b \widehat{e}   
+ 8 d c b a \widehat{e}   
\\
{} 
+ 4 e a b c \widehat{d}   
+ 4 e b a c \widehat{d}   
+ 4 e c a b \widehat{d}   
+ 4 e c b a \widehat{d}   
+ 2 e d a b \widehat{c}   
+ 2 e d b a \widehat{c}   
+   e d c a \widehat{b}   
+   e d c b \widehat{a},  
\\[3pt]
X( ((ab)(cd))e ) =
\\
\;\;\;\;\,
  2 \widehat{a} b c d e   
+ 2 \widehat{a} b d c e   
+ 2 \widehat{b} a c d e   
+ 2 \widehat{b} a d c e   
+ 2 \widehat{c} d a b e   
+ 2 \widehat{c} d b a e   
+ 2 \widehat{d} c a b e   
+ 2 \widehat{d} c b a e   
\\
{} 
+ 8 \widehat{e} a b c d   
+ 8 \widehat{e} a b d c   
+ 8 \widehat{e} b a c d   
+ 8 \widehat{e} b a d c   
+ 8 \widehat{e} c d a b   
+ 8 \widehat{e} c d b a   
+ 8 \widehat{e} d c a b   
+ 8 \widehat{e} d c b a   
\\
{} 
+ 2 a \widehat{b} c d e   
+ 2 a \widehat{b} d c e   
+ 2 b \widehat{a} c d e   
+ 2 b \widehat{a} d c e   
+ 2 c \widehat{d} a b e   
+ 2 c \widehat{d} b a e   
+ 2 d \widehat{c} a b e   
+ 2 d \widehat{c} b a e   
\\
{} 
+ 2 e \widehat{a} b c d   
+ 2 e \widehat{a} b d c   
+ 2 e \widehat{b} a c d   
+ 2 e \widehat{b} a d c   
+ 2 e \widehat{c} d a b   
+ 2 e \widehat{c} d b a   
+ 2 e \widehat{d} c a b   
+ 2 e \widehat{d} c b a   
\\
{} 
+ 2 a b \widehat{c} d e   
+ 2 a b \widehat{d} c e   
+ 2 b a \widehat{c} d e   
+ 2 b a \widehat{d} c e   
+ 2 c d \widehat{a} b e   
+ 2 c d \widehat{b} a e   
+ 2 d c \widehat{a} b e   
+ 2 d c \widehat{b} a e   
\\
{} 
+ 2 e a \widehat{b} c d   
+ 2 e a \widehat{b} d c   
+ 2 e b \widehat{a} c d   
+ 2 e b \widehat{a} d c   
+ 2 e c \widehat{d} a b   
+ 2 e c \widehat{d} b a   
+ 2 e d \widehat{c} a b   
+ 2 e d \widehat{c} b a   
\\
{} 
+ 2 a b c \widehat{d} e   
+ 2 a b d \widehat{c} e   
+ 2 b a c \widehat{d} e   
+ 2 b a d \widehat{c} e   
+ 2 c d a \widehat{b} e   
+ 2 c d b \widehat{a} e   
+ 2 d c a \widehat{b} e   
+ 2 d c b \widehat{a} e   
\\
{} 
+ 2 e a b \widehat{c} d   
+ 2 e a b \widehat{d} c   
+ 2 e b a \widehat{c} d   
+ 2 e b a \widehat{d} c   
+ 2 e c d \widehat{a} b   
+ 2 e c d \widehat{b} a   
+ 2 e d c \widehat{a} b   
+ 2 e d c \widehat{b} a   
\\
{} 
+ 8 a b c d \widehat{e}   
+ 8 a b d c \widehat{e}   
+ 8 b a c d \widehat{e}   
+ 8 b a d c \widehat{e}   
+ 8 c d a b \widehat{e}   
+ 8 c d b a \widehat{e}   
+ 8 d c a b \widehat{e}   
+ 8 d c b a \widehat{e}   
\\
{} 
+ 2 e a b c \widehat{d}   
+ 2 e a b d \widehat{c}   
+ 2 e b a c \widehat{d}   
+ 2 e b a d \widehat{c}   
+ 2 e c d a \widehat{b}   
+ 2 e c d b \widehat{a}   
+ 2 e d c a \widehat{b}   
+ 2 e d c b \widehat{a},  
\\[3pt]
X( ((ab)c)(de) ) =
\\
\;\;\;\;\,
  2 \widehat{a} b c d e   
+ 2 \widehat{a} b c e d   
+ 2 \widehat{b} a c d e   
+ 2 \widehat{b} a c e d   
+ 4 \widehat{c} a b d e   
+ 4 \widehat{c} a b e d   
+ 4 \widehat{c} b a d e   
+ 4 \widehat{c} b a e d   
\\
{} 
+ 4 \widehat{d} e a b c   
+ 4 \widehat{d} e b a c   
+ 4 \widehat{d} e c a b   
+ 4 \widehat{d} e c b a   
+ 4 \widehat{e} d a b c   
+ 4 \widehat{e} d b a c   
+ 4 \widehat{e} d c a b   
+ 4 \widehat{e} d c b a   
\\
{} 
+ 2 a \widehat{b} c d e   
+ 2 a \widehat{b} c e d   
+ 2 b \widehat{a} c d e   
+ 2 b \widehat{a} c e d   
+ 2 c \widehat{a} b d e   
+ 2 c \widehat{a} b e d   
+ 2 c \widehat{b} a d e   
+ 2 c \widehat{b} a e d   
\\
{} 
+ 4 d \widehat{e} a b c   
+ 4 d \widehat{e} b a c   
+ 4 d \widehat{e} c a b   
+ 4 d \widehat{e} c b a   
+ 4 e \widehat{d} a b c   
+ 4 e \widehat{d} b a c   
+ 4 e \widehat{d} c a b   
+ 4 e \widehat{d} c b a   
\\
{} 
+ 4 a b \widehat{c} d e   
+ 4 a b \widehat{c} e d   
+ 4 b a \widehat{c} d e   
+ 4 b a \widehat{c} e d   
+ 2 c a \widehat{b} d e   
+ 2 c a \widehat{b} e d   
+ 2 c b \widehat{a} d e   
+ 2 c b \widehat{a} e d   
\\
{} 
+ 2 d e \widehat{a} b c   
+ 2 d e \widehat{b} a c   
+ 4 d e \widehat{c} a b   
+ 4 d e \widehat{c} b a   
+ 2 e d \widehat{a} b c   
+ 2 e d \widehat{b} a c   
+ 4 e d \widehat{c} a b   
+ 4 e d \widehat{c} b a   
\\
{} 
+ 4 a b c \widehat{d} e   
+ 4 a b c \widehat{e} d   
+ 4 b a c \widehat{d} e   
+ 4 b a c \widehat{e} d   
+ 4 c a b \widehat{d} e   
+ 4 c a b \widehat{e} d   
+ 4 c b a \widehat{d} e   
+ 4 c b a \widehat{e} d   
\\
{} 
+ 2 d e a \widehat{b} c   
+ 2 d e b \widehat{a} c   
+ 2 d e c \widehat{a} b   
+ 2 d e c \widehat{b} a   
+ 2 e d a \widehat{b} c   
+ 2 e d b \widehat{a} c   
+ 2 e d c \widehat{a} b   
+ 2 e d c \widehat{b} a   
\\
{} 
+ 4 a b c d \widehat{e}   
+ 4 a b c e \widehat{d}   
+ 4 b a c d \widehat{e}   
+ 4 b a c e \widehat{d}   
+ 4 c a b d \widehat{e}   
+ 4 c a b e \widehat{d}   
+ 4 c b a d \widehat{e}   
+ 4 c b a e \widehat{d}   
\\
{} 
+ 4 d e a b \widehat{c}   
+ 4 d e b a \widehat{c}   
+ 2 d e c a \widehat{b}   
+ 2 d e c b \widehat{a}   
+ 4 e d a b \widehat{c}   
+ 4 e d b a \widehat{c}   
+ 2 e d c a \widehat{b}   
+ 2 e d c b \widehat{a}.  
\end{array}$
\caption{Expansions with identity permutation in degree 5}
\label{expansions5}
\end{figure}

\begin{lemma}
Every multilinear polynomial identity of degree $n \le 5$ satisfied by the symmetrization of 
the Jordan diproduct is a consequence of commutativity.
\end{lemma}

\begin{proof}
Similar to Lemma \ref{lemma4}, but the matrix is much larger.
In degree 5, the three association types 
$(((--)-)-)-$, $((--)(--))-$, $((--)-)(--)$
for a commutative nonassociative operation
have respectively 60, 15, 30 multilinear monomials.
We expand the monomials with the identity permutation into the diassociative operad
(Figure \ref{expansions5});
there are $5 \cdot 5! = 600$ multilinear diassociative monomials.
We construct the $600 \times 105$ matrix $E$ whose $(i,j)$ entry is the coefficient of 
the $i$-th diassociative monomial in the expansion of the $j$-th commutative nonassociative monomial.
We extract the submatrix consisting of these 105 rows:
1--69, 71, 73--83, 85, 86, 91, 97--101, 103, 121--125, 127--129, 133, 145, 147, 149, 151, 169, 241.
These rows are the lexicographically first subset which forms a basis of the row space.
This submatrix has full rank and hence so does $E$, which completes the proof.
\end{proof}


\section{Degree 6}

\begin{theorem}
\label{theorem6}
In degree 6, there is an 8-dimensional space of multilinear polynomial identities satisfied by 
the symmetrization of the Jordan diproduct which do not follow from commutativity.
This $S_6$-module has the structure $3[6] \oplus [51]$, where $[\lambda]$ denotes the simple
module corresponding to partition $\lambda$, and so we may restrict our attention to
nonlinear identities in the variables $x^6$ and $x^5y$.
Every identity in the variables $x^6$ is a linear combination of these three identities:
\begin{align}
&
\label{deg6.6.1}
    ((x^2x^2)x)x 
-2  ((x^2x)x^2)x 
+4  ((x^2x)x)x^2 
-3  (x^2x^2)x^2 
\equiv 0,
\\[-2pt]
&
\label{deg6.6.2}
 2  ((x^2x^2)x)x 
-2  ((x^2x)x^2)x 
-2  ((x^2x)x)x^2 
+   (x^2x^2)x^2 
+   (x^2x)(x^2x) 
\equiv 0,
\\[-2pt]
&
\label{deg6.6.3}
 8  (((x^2x)x)x)x 
-4  ((x^2x^2)x)x 
-5  ((x^2x)x^2)x 
-   ((x^2x)x)x^2 
-   (x^2x^2)x^2 
\\[-2pt]
\notag
&\quad
+3  (x^2x)(x^2x) 
\equiv 0.
\end{align}
Every identity in $x^5y$ is a linear combination of the four identities
in Figure \ref{fouridentities}.
\end{theorem}

\begin{figure}[ht]
\begin{align}
\label{deg6.51.1}
&
     ((x^2x^2)x)y    
 +   ((x^2x^2)y)x    
 +4  ((x^2(xy))x)x   
 -   ((x^2x)x^2)y    
\\
\notag
&\quad
 -2  ((x^2x)(xy))x   
 -   ((x^2y)x^2)x    
 -2  (((xy)x)x^2)x   
 -2  ((x^2x)x)(xy)   
\\
\notag
&\quad
 -   ((x^2x)y)x^2    
 -   ((x^2y)x)x^2    
 -2  (((xy)x)x)x^2   
 +   (x^2x^2)(xy)    
\\
\notag
&\quad
 +2  (x^2(xy))x^2    
 +   (x^2x)(x^2y)    
 +2  (x^2x)((xy)x) \equiv 0,  
\\
\label{deg6.51.2}
&
  8  (((x^2x)x)y)x   
 +8  (((x^2y)x)x)x   
 -5  ((x^2x^2)y)x    
 -4  ((x^2(xy))x)x   
\\
\notag
&\quad
 -   ((x^2x)x^2)y    
 -6  ((x^2x)(xy))x   
 -3  ((x^2y)x^2)x    
 +2  ((x^2x)x)(xy)   
\\
\notag
&\quad
 +   ((x^2x)y)x^2    
 -3  ((x^2y)x)x^2    
 +4  (((xy)x)x)x^2   
 -2  (x^2x^2)(xy)    
\\
\notag
&\quad
 -4  (x^2(xy))x^2    
 +3  (x^2x)(x^2y)    
 +2  (x^2x)((xy)x) \equiv 0,  
\\
\label{deg6.51.3}
&
  2  ((x^2x^2)x)y    
 +2  ((x^2x^2)y)x    
 +8  ((x^2(xy))x)x   
 -3  ((x^2x)x^2)y    
\\
\notag
&\quad
 -6  ((x^2x)(xy))x   
 -3  ((x^2y)x^2)x    
 -6  (((xy)x)x^2)x   
 +6  ((x^2x)x)(xy)   
\\
\notag
&\quad
 +3  ((x^2x)y)x^2    
 +3  ((x^2y)x)x^2    
 +6  (((xy)x)x)x^2   
 -5  (x^2x^2)(xy)    
\\
\notag
&\quad
-10  (x^2(xy))x^2    
 +   (x^2x)(x^2y)    
 +2  (x^2x)((xy)x) \equiv 0,  
\\
\label{deg6.51.4}
&
  8  (((x^2x)x)x)y   
 +8  (((x^2x)y)x)x   
+16  ((((xy)x)x)x)x  
 -5  ((x^2x^2)x)y    
\\
\notag
&\quad
-16  ((x^2(xy))x)x   
 -3  ((x^2x)x^2)y    
 -2  ((x^2x)(xy))x   
 -   ((x^2y)x^2)x    
\\
\notag
&\quad
 -8  (((xy)x)x^2)x   
 -2  ((x^2x)x)(xy)   
 -   ((x^2x)y)x^2    
 +3  ((x^2y)x)x^2    
\\
\notag
&\quad
 -4  (((xy)x)x)x^2   
 -   (x^2x^2)(xy)    
 -2  (x^2(xy))x^2    
 +2  (x^2x)(x^2y)    
\\
\notag
&\quad
 +8  (x^2x)((xy)x) \equiv 0.  
\end{align}
\caption{Four identities in the variables $x^5y$}
\label{fouridentities}
\end{figure}

\begin{proof}
At this point, the matrices become so large that efficient computation requires modular arithmetic,
and we use the large prime $p = 1000003$.

There are six association types for a commutative nonassociative operation, with respectively 
360, 90, 180, 180, 45, 90 multilinear monomials, for a total of 945:
\[
\begin{array}{ccc}
((((--)-)-)-)-, &\qquad
(((--)(--))-)-, &\qquad
(((--)-)(--))-,
\\
(((--)-)-)(--), &\qquad
((--)(--))(--), &\qquad
((--)-)((--)-).
\end{array}
\]
With arguments $abcdef$, the expansions of these association types into the diassociative operad 
(using the symmetrized Jordan diproduct) each have 192 terms (and are therefore omitted); 
the coefficients of the expansions belong respectively to the sets
$\{1,2,4,8,16\}$, 
$\{2,8,16\}$, 
$\{2,4,16\}$, 
$\{2,4,8\}$, 
$\{4,8\}$, 
$\{4,8\}$.
There are $6 \cdot 6! = 4320$ multilinear diassociative monomials.
We construct the $4320 \times 945$ matrix $E$ whose $(i,j)$ entry is the coefficient of 
the $i$-th diassociative monomial in the expansion of the $j$-th commutative nonassociative monomial.
Using linear algebra over $\mathbb{F}_p$, we find that $\mathrm{rank}(E) = 937$ 
and hence the nullspace has dimension 8.

Let $N$ denote the unique $8 \times 945$ matrix in row canonical form (RCF) over $\mathbb{F}_p$ 
whose row space is the nullspace of $E$.
The entries of $N$ belong to the following (extremely short) list of 33 congruence classes:
0--3, 8, 13, 14, 250000--250002, 499996--500004, 500012, 749997--750004, 999998--1000002.
This strongly suggests that if we had been able to do this calculation using rational arithmetic
then the denominators would all be divisors of 4.
We multiply each row of $N$ by 4, express the entries using symmetric representatives modulo $p$, 
interpret these symmetric representatives as integers, and divide each row by the GCD of its entries.
The results appear in Figure \ref{deg6int}, where column 2 gives the number of nonzero entries for
each row; since these numbers are so large, it is not practical to display the corresponding multilinear 
polynomial identities.

\begin{figure}[ht]
\begin{tabular}{c@{\qquad}c@{\qquad}l}
row & nonzero & reconstructed integer entries \\ \midrule
1 & 693 & $-18, -17, -14, -13, -12, -4, 0, 1, 2, 3, 4, 5, 6, 8, 10, 42$ \\
2 & 660 & $-10, -8, -4, -3, -2, 0, 1, 2, 4, 10$ \\
3 & 711 & $-8, -6, -4, -3, -2, -1, 0, 1, 2, 3, 4$ \\
4 & 684 & $-22, -21, -20, -18, -17, -16, -4, -2, 0, 1, 3, 4, 5, 6, 7, 8, 52$ \\
5 & 585 & $-9, -8, -5, -4, -3, -1, 0, 2, 3, 4, 5, 12$ \\
6 & 585 & $-9, -8, -5, -4, -3, -1, 0, 2, 3, 4, 5, 12$ \\
7 & 405 & $-3, 0, 1, 8$ \\
8 & 495 & $-5, 0, 1, 14$
\end{tabular}
\caption{Integer coefficients reconstructed from modular computations}
\label{deg6int}
\end{figure}

\begin{figure}[ht]
\footnotesize
$\begin{array}{l}
\left[
\begin{array}{@{\,}c@{\,}c@{\,}c@{\,}c@{\,}c@{\,}c@{\,}c@{\,}c@{\,}}
1 & . & . & . & . & . & . & . \\
. & 1 & . & . & . & . & . & . \\
. & . & 1 & . & . & . & . & . \\
. & . & . & 1 & . & . & . & . \\
. & . & . & . & 1 & . & . & . \\
. & . & . & . & . & 1 & . & . \\
. & . & . & . & . & . & 1 & . \\
. & . & . & . & . & . & . & 1
\end{array}
\right]
\;
\left[
\begin{array}{@{\,}c@{\,}c@{\,}c@{\,}c@{\,}c@{\,}c@{\,}c@{\,}c@{\,}}
1 & . & . & . & . & . & . & . \\
. & 1 & . & . & . & . & . & . \\
. & . & 1 & . & . & . & . & . \\
. & . & . & 1 & . & . & . & . \\
. & . & . & . & . & 1 & . & . \\
. & . & . & . & 1 & . & . & . \\
. & . & . & . & . & . & 1 & . \\
. & . & . & . & . & . & . & 1
\end{array}
\right]
\;
\left[
\begin{array}{@{\,}c@{\,}c@{\,}c@{\,}c@{\,}c@{\,}c@{\,}c@{\,}c@{\,}}
. & . & . & 1 & . & . & . & . \\
-1 & 1 & . & 1 & . & . & . & . \\
. & . & 1 & . & . & . & . & . \\
1 & . & . & . & . & . & . & . \\
. & . & . & . & . & 1 & . & . \\
. & . & . & . & 1 & . & . & . \\
. & . & . & . & . & . & 1 & . \\
. & . & . & . & . & . & . & 1
\end{array}
\right]
\;
\left[
\begin{array}{@{\,}c@{\,}c@{\,}c@{\,}c@{\,}c@{\,}c@{\,}c@{\,}c@{\,}}
-1 & 1 & . & 1 & . & . & . & . \\
. & . & . & 1 & . & . & . & . \\
. & . & 1 & . & . & . & . & . \\
. & 1 & . & . & . & . & . & . \\
-1 & 1 & . & . & . & 1 & . & . \\
-1 & 1 & . & . & 1 & . & . & . \\
10 & -10 & . & . & . & . & 1 & . \\
2 & -2 & . & . & . & . & . & 1
\end{array}
\right]
\;
\left[
\begin{array}{@{\,}c@{\,}c@{\,}c@{\,}c@{\,}c@{\,}c@{\,}c@{\,}c@{\,}}
. & . & . & . & 1 & . & . & . \\
-1 & 1 & . & . & 1 & . & . & . \\
-1 & . & 1 & . & 1 & . & . & . \\
. & . & . & 1 & . & . & . & . \\
. & . & . & . & . & 1 & . & . \\
1 & . & . & . & . & . & . & . \\
. & . & . & . & . & . & 1 & . \\
3 & . & . & . & -3 & . & . & 1
\end{array}
\right]
\\
\\
\left[
\begin{array}{@{\,}c@{\,}c@{\,}c@{\,}c@{\,}c@{\,}c@{\,}c@{\,}c@{\,}}
-1 & . & 1 & . & 1 & . & . & . \\
-1 & 1 & . & . & 1 & . & . & . \\
. & . & . & . & 1 & . & . & . \\
. & . & . & 1 & . & . & . & . \\
-1 & . & 1 & . & . & 1 & . & . \\
. & . & 1 & . & . & . & . & . \\
8 & . & -8 & . & . & . & 1 & . \\
3 & . & . & . & -3 & . & . & 1
\end{array}
\right]
\;
\left[
\begin{array}{@{\,}c@{\,}c@{\,}c@{\,}c@{\,}c@{\,}c@{\,}c@{\,}c@{\,}}
-1 & . & 1 & . & 1 & . & . & . \\
. & . & . & . & 1 & . & . & . \\
-1 & 1 & . & . & 1 & . & . & . \\
-1 & 1 & . & 1 & . & . & . & . \\
-1 & . & 1 & . & . & 1 & . & . \\
. & . & 1 & . & . & . & . & . \\
10 & -2 & -8 & . & . & . & 1 & . \\
5 & -2 & . & . & -3 & . & . & 1
\end{array}
\right]
\;
\left[
\begin{array}{@{\,}c@{\,}c@{\,}c@{\,}c@{\,}c@{\,}c@{\,}c@{\,}c@{\,}}
. & . & . & . & 1 & . & . & . \\
-1 & 1 & . & . & 1 & . & . & . \\
. & . & 1 & -1 & 1 & . & . & . \\
1 & . & . & . & . & . & . & . \\
. & . & . & . & . & 1 & . & . \\
. & . & . & 1 & . & . & . & . \\
. & . & . & . & . & . & 1 & . \\
. & . & . & 3 & -3 & . & . & 1
\end{array}
\right]
\;
\left[
\begin{array}{@{\,}c@{\,}c@{\,}c@{\,}c@{\,}c@{\,}c@{\,}c@{\,}c@{\,}}
-1 & 1 & . & . & 1 & . & . & . \\
. & . & . & . & 1 & . & . & . \\
. & . & 1 & -1 & 1 & . & . & . \\
. & 1 & . & . & . & . & . & . \\
-1 & 1 & . & . & . & 1 & . & . \\
-1 & 1 & . & 1 & . & . & . & . \\
10 & -10 & . & . & . & . & 1 & . \\
2 & -2 & . & 3 & -3 & . & . & 1
\end{array}
\right]
\\
\\
\left[
\begin{array}{@{\,}c@{\,}c@{\,}c@{\,}c@{\,}c@{\,}c@{\,}c@{\,}c@{\,}}
-1 & . & 1 & . & 1 & . & . & . \\
-1 & 1 & . & . & 1 & . & . & . \\
. & . & 1 & -1 & 1 & . & . & . \\
. & . & 1 & . & . & . & . & . \\
-1 & . & 1 & . & . & 1 & . & . \\
. & . & . & 1 & . & . & . & . \\
8 & . & -8 & . & . & . & 1 & . \\
3 & . & -3 & 3 & -3 & . & . & 1
\end{array}
\right]
\;
\left[
\begin{array}{@{\,}c@{\,}c@{\,}c@{\,}c@{\,}c@{\,}c@{\,}c@{\,}c@{\,}}
-1 &  . &  1 &  . &  1 & . & . & . \\
 . &  . &  . &  . &  1 & . & . & . \\
 . &  . &  1 & -1 &  1 & . & . & . \\
 . &  . &  1 &  . &  . & . & . & . \\
-1 &  . &  1 &  . &  . & 1 & . & . \\
-1 &  1 &  . &  1 &  . & . & . & . \\
10 & -2 & -8 &  . &  . & . & 1 & . \\
 2 &  1 & -3 &  3 & -3 & . & . & 1
\end{array}
\right]
\end{array}$
\caption{Representation matrices for conjugacy class representatives}
\label{10matrices}
\end{figure}

The symmetric group $S_6$ acts on the row space of $N$ by permuting $abcdef$.
For the partitions
$1^6$, $21^4$, $2^21^2$, $2^3$, $31^3$, $321$, $3^2$, $41^2$, $42$, $51$, $6$
we choose the following conjugacy class representatives, expressed as products of disjoint cycles: 
$e$, (12), (12)(34), (12)(34)(56), (123), (123)(45), (123)(456), (1234), (1234)(56), (12345), (123456).
For each of these 11 permutations $\sigma$, we compute the matrix representing $\sigma$
with respect to the reconstructed integer basis of the row space of $N$.
These 11 matrices are displayed in Figure \ref{10matrices}.
The traces of these matrices give the character of the row space of $N$ as an $S_6$-module:
$[ 8, 6, 4, 2, 5, 3, 2, 4, 2, 3, 2 ]$.
This character is a linear combination of the first two rows of the
character table for $S_6$:
\[
3
[1,1,1,1,1,1,1,1,1,1,1]
+
[5,3,1,-1,2,0,-1,1,-1,0,-1].
\]
Hence the row space of $N$ is isomorphic to $3[6] \oplus [51]$ as an $S_6$-module.

We now consider the problem of finding all nonlinear identities in the variables $x^6$ and $x^5y$.
The structure of the row space of $N$ as an $S_6$-module implies that every multilinear identity
in degree 6 can be expressed as a linear combination of linearizations of nonlinear identities 
corresponding to these two partitions.
Since the matrices are small, we are able to use integer arithmetic.

\begin{figure}[ht]
$
E =
\left[
\begin{array}{r@{\;\;\;}r@{\;\;\;}r@{\;\;\;}r@{\;\;\;}r@{\;\;\;}r}
342 & 336 & 312 & 216 & 192 & 192 \\
118 & 112 & 104 & 168 & 192 & 128 \\
 52 &  64 &  96 & 128 & 128 & 192 \\
 52 &  64 &  96 & 128 & 128 & 192 \\
118 & 112 & 104 & 168 & 192 & 128 \\
342 & 336 & 312 & 216 & 192 & 192
\end{array}
\right]
\quad
U =
\left[
\begin{array}{r@{\;\;\;}r@{\;\;\;}r@{\;\;\;}r@{\;\;\;}r@{\;\;\;}r}
 1 & -2 &  0 & -2 & -1 &  5 \\
-4 &  1 &  2 &  1 & -2 &  3 \\
 0 & -1 &  0 & -2 & -1 &  5 \\
 0 &  2 & -2 & -2 &  1 &  1 \\
 0 &  1 & -2 &  4 & -3 &  0 \\
-8 &  4 &  5 &  1 &  1 & -3
\end{array}
\right]
$
\caption{Expansion matrix and transform matrix for variables $x^6$}
\label{deg6exp6}
\end{figure}

With variables $x^6$, there is only one commutative nonassociative monomial for each of the six
association types, and only one diassociative monomial for each possible center; we obtain the expansion 
matrix $E$ of Figure \ref{deg6exp6}.
Using elementary row operations defined over $\mathbb{Z}$, together with the LLL algorithm for 
lattice basis reduction \cite{BP2009}, we compute the integer matrix $U$ of Figure \ref{deg6exp6} 
for which $\det(U) = \pm 1$ and $U E^t = H$ where $H$ (not displayed) is the Hermite normal form (HNF) 
of the transpose of $E$.
From this we find that $\mathrm{rank}(E) = 3$ and that the last three rows of $U$ form a reduced basis
for the integer nullspace of $E$.
The corresponding polynomial identities are displayed in equations \eqref{deg6.6.1}--\eqref{deg6.6.3}.

With variables $x^5y$, there are respectively 5, 3, 4, 4, 2, 2 commutative nonassociative monomials
for each of the six association types, which we order as follows:
\[
\begin{array}{lllll}
(((x^2x)x)x)y, &
(((x^2x)x)y)x, &
(((x^2x)y)x)x, &
(((x^2y)x)x)x, &
((((xy)x)x)x)x, 
\\
((x^2x^2)x)y, &
((x^2x^2)y)x, &
((x^2(xy))x)x, &
((x^2x)x^2)y, &
((x^2x)(xy))x, 
\\
((x^2y)x^2)x, &
(((xy)x)x^2)x, &
((x^2x)x)(xy), &
((x^2x)y)x^2, &
((x^2y)x)x^2, 
\\
(((xy)x)x)x^2, &
(x^2x^2)(xy), &
(x^2(xy))x^2, &
(x^2x)(x^2y), &
(x^2x)((xy)x).
\end{array}
\]
There are 36 diassociative monomials, ordered by position of the center and then by the position of $y$.
We obtain the expansion matrix $E$ of Figure \ref{deg6exp51}.
Proceeding as before, we find that $\mathrm{rank}(E) = 4$ and that the last three rows of the transform
matrix $U$ form a reduced basis for the integer nullspace of $E$.
The corresponding polynomial identities are displayed in equations \eqref{deg6.51.1}--\eqref{deg6.51.4}.
\end{proof}

\begin{figure}[ht]
\small
$
\left[
\begin{array}{
r@{\;\;}r@{\;\;}r@{\;\;}r@{\;\;}r@{\;\;}r@{\;\;}r@{\;\;}r@{\;\;}r@{\;\;}r@{\;\;}
r@{\;\;}r@{\;\;}r@{\;\;}r@{\;\;}r@{\;\;}r@{\;\;}r@{\;\;}r@{\;\;}r@{\;\;}r
}
86 & 128 & 64 & 32 & 16 & 80 & 128 & 32 & 56 & 64 & 64 & 32 & 44 & 64 & 32 & 16 & 32 & 32 & 48 & 24 \\
0 & 22 & 96 & 80 & 72 & 0 & 16 & 80 & 0 & 76 & 16 & 72 & 44 & 0 & 32 & 48 & 32 & 32 & 0 & 48 \\
0 & 0 & 6 & 88 & 124 & 0 & 0 & 84 & 0 & 12 & 128 & 80 & 0 & 24 & 64 & 64 & 0 & 48 & 48 & 24 \\
0 & 0 & 64 & 82 & 98 & 0 & 0 & 84 & 0 & 64 & 24 & 80 & 0 & 64 & 40 & 56 & 0 & 48 & 32 & 32 \\
0 & 128 & 96 & 56 & 31 & 0 & 128 & 52 & 0 & 80 & 64 & 44 & 64 & 0 & 32 & 28 & 64 & 16 & 0 & 48 \\
256 & 64 & 16 & 4 & 1 & 256 & 64 & 4 & 256 & 16 & 16 & 4 & 64 & 64 & 16 & 4 & 64 & 16 & 64 & 16 \\
32 & 22 & 32 & 16 & 8 & 32 & 16 & 16 & 48 & 12 & 16 & 8 & 20 & 64 & 32 & 16 & 32 & 32 & 32 & 16 \\
0 & 10 & 12 & 32 & 32 & 0 & 16 & 24 & 0 & 20 & 16 & 24 & 20 & 0 & 32 & 48 & 32 & 32 & 0 & 32 \\
0 & 0 & 4 & 22 & 46 & 0 & 0 & 28 & 0 & 8 & 24 & 32 & 0 & 16 & 40 & 56 & 0 & 48 & 32 & 16 \\
0 & 0 & 32 & 34 & 26 & 0 & 0 & 28 & 0 & 16 & 24 & 24 & 0 & 64 & 40 & 32 & 0 & 48 & 32 & 16 \\
0 & 64 & 32 & 12 & 5 & 0 & 64 & 12 & 0 & 32 & 16 & 12 & 64 & 0 & 16 & 12 & 64 & 16 & 0 & 32 \\
86 & 22 & 6 & 2 & 1 & 80 & 16 & 4 & 56 & 16 & 8 & 4 & 64 & 24 & 8 & 4 & 64 & 16 & 32 & 16 \\
20 & 10 & 6 & 8 & 4 & 32 & 16 & 4 & 48 & 8 & 16 & 8 & 20 & 24 & 32 & 16 & 32 & 16 & 48 & 24 \\
0 & 10 & 8 & 6 & 14 & 0 & 16 & 12 & 0 & 20 & 8 & 24 & 20 & 0 & 8 & 40 & 32 & 16 & 0 & 48 \\
0 & 0 & 6 & 14 & 16 & 0 & 0 & 16 & 0 & 12 & 24 & 24 & 0 & 24 & 40 & 32 & 0 & 32 & 48 & 24 \\
0 & 0 & 16 & 16 & 10 & 0 & 0 & 16 & 0 & 16 & 32 & 16 & 0 & 64 & 32 & 16 & 0 & 32 & 64 & 16 \\
0 & 22 & 12 & 6 & 6 & 0 & 16 & 12 & 0 & 28 & 8 & 16 & 44 & 0 & 8 & 16 & 32 & 16 & 0 & 48 \\
32 & 10 & 4 & 2 & 2 & 32 & 16 & 4 & 48 & 12 & 8 & 8 & 44 & 16 & 8 & 8 & 32 & 16 & 32 & 32 \\
32 & 10 & 4 & 2 & 2 & 32 & 16 & 4 & 48 & 12 & 8 & 8 & 44 & 16 & 8 & 8 & 32 & 16 & 32 & 32 \\
0 & 22 & 12 & 6 & 6 & 0 & 16 & 12 & 0 & 28 & 8 & 16 & 44 & 0 & 8 & 16 & 32 & 16 & 0 & 48 \\
0 & 0 & 16 & 16 & 10 & 0 & 0 & 16 & 0 & 16 & 32 & 16 & 0 & 64 & 32 & 16 & 0 & 32 & 64 & 16 \\
0 & 0 & 6 & 14 & 16 & 0 & 0 & 16 & 0 & 12 & 24 & 24 & 0 & 24 & 40 & 32 & 0 & 32 & 48 & 24 \\
0 & 10 & 8 & 6 & 14 & 0 & 16 & 12 & 0 & 20 & 8 & 24 & 20 & 0 & 8 & 40 & 32 & 16 & 0 & 48 \\
20 & 10 & 6 & 8 & 4 & 32 & 16 & 4 & 48 & 8 & 16 & 8 & 20 & 24 & 32 & 16 & 32 & 16 & 48 & 24 \\
86 & 22 & 6 & 2 & 1 & 80 & 16 & 4 & 56 & 16 & 8 & 4 & 64 & 24 & 8 & 4 & 64 & 16 & 32 & 16 \\
0 & 64 & 32 & 12 & 5 & 0 & 64 & 12 & 0 & 32 & 16 & 12 & 64 & 0 & 16 & 12 & 64 & 16 & 0 & 32 \\
0 & 0 & 32 & 34 & 26 & 0 & 0 & 28 & 0 & 16 & 24 & 24 & 0 & 64 & 40 & 32 & 0 & 48 & 32 & 16 \\
0 & 0 & 4 & 22 & 46 & 0 & 0 & 28 & 0 & 8 & 24 & 32 & 0 & 16 & 40 & 56 & 0 & 48 & 32 & 16 \\
0 & 10 & 12 & 32 & 32 & 0 & 16 & 24 & 0 & 20 & 16 & 24 & 20 & 0 & 32 & 48 & 32 & 32 & 0 & 32 \\
32 & 22 & 32 & 16 & 8 & 32 & 16 & 16 & 48 & 12 & 16 & 8 & 20 & 64 & 32 & 16 & 32 & 32 & 32 & 16 \\
256 & 64 & 16 & 4 & 1 & 256 & 64 & 4 & 256 & 16 & 16 & 4 & 64 & 64 & 16 & 4 & 64 & 16 & 64 & 16 \\
0 & 128 & 96 & 56 & 31 & 0 & 128 & 52 & 0 & 80 & 64 & 44 & 64 & 0 & 32 & 28 & 64 & 16 & 0 & 48 \\
0 & 0 & 64 & 82 & 98 & 0 & 0 & 84 & 0 & 64 & 24 & 80 & 0 & 64 & 40 & 56 & 0 & 48 & 32 & 32 \\
0 & 0 & 6 & 88 & 124 & 0 & 0 & 84 & 0 & 12 & 128 & 80 & 0 & 24 & 64 & 64 & 0 & 48 & 48 & 24 \\
0 & 22 & 96 & 80 & 72 & 0 & 16 & 80 & 0 & 76 & 16 & 72 & 44 & 0 & 32 & 48 & 32 & 32 & 0 & 48 \\
86 & 128 & 64 & 32 & 16 & 80 & 128 & 32 & 56 & 64 & 64 & 32 & 44 & 64 & 32 & 16 & 32 & 32 & 48 & 24
\end{array}
\right]
$
\caption{Expansion matrix $E$ for variables $x^5y$}
\label{deg6exp51}
\end{figure}


\section{Degree 7}

\begin{theorem}
In degree 7, there is a 570-dimensional space of multilinear polynomial identities satisfied by 
the symmetrization of the Jordan diproduct which do not follow from commutativity or 
the linearizations of the identities in Theorem \ref{theorem6}.
This $S_7$-module has the following structure:
  \[
  7 [61] \oplus
  6 [52] \oplus
  4 [51^2] \oplus
  5 [43] \oplus
  5 [421] \oplus
  [41^3] \oplus
  3 [3^21] \oplus
  [32^2] \oplus
  [321^2].
  \]
\end{theorem}

\begin{proof}
We obtain this result by extension of the methods already described for degree $n \le 6$.
Owing to the large size of the matrices in degree 7, we decompose the computation into 
subproblems corresponding to the irreducible representations of the symmetric group $S_7$.
The theory and algorithms required for this application of representation theory to polynomial
identities are explained in detail in the survey paper \cite{BMP}.
We provide a brief summary of the computations.

Step 1.
There are 11 association types for a commutative nonassociative operation in degree 7,
on which we impose the reverse degree-lexicographical order:
\[
\begin{array}{l@{\qquad}l@{\qquad}l}
(((((--)-)-)-)-)-, &
((((--)(--))-)-)-, &
((((--)-)(--))-)-, \\
((((--)-)-)(--))-, &
(((--)(--))(--))-, &
(((--)-)((--)-))-, \\
((((--)-)-)-)(--), &
(((--)(--))-)(--), &
(((--)-)(--))(--), \\
(((--)-)-)((--)-), &
((--)(--))((--)-). &
\end{array}
\]
For each type $t$, there is a nonempty set of permutations $\sigma \in S_7$ of order 2 for which
$t(\iota) = t(\sigma)$ as a result of the commutativity of the symmetrized Jordan diproduct,
where $\iota$ is the identity permutation and $t(\sigma)$ denotes type $t$ applied to permutation $\sigma$ 
of (the subscripts of) the arguments $x_1, \dots, x_7$.
For example, the last type produces $\sigma = (12)$, $(34)$, $(13)(24)$, $(56)$.
We obtain 30 polynomial identities $t(\iota) - t(\sigma) \equiv 0$ in degree 7,
called the \emph{symmetries of the association types}.

Step 2.
We linearize the nonlinear identities of Theorem \ref{theorem6}.
For variables $x^6$, we replace the six occurrences of $x$ in each term by $x_1 \cdots x_6$,
and sum over all permutations of the subscripts $1,\dots,6$.
For variables $x^5y$, we replace the five occurrences of $x$ in each term by $x_1 \cdots x_5$,
sum over all permutations of subscripts $1,\dots,5$, and replace $y$ by $x_6$.
In both cases, we straighten the resulting monomials using commutativity of the symmetrized Jordan 
diproduct: we replace each monomial by the lexicographically first representative of its
equivalence class.

Step 3.
Let $f(x_1,\dots,x_6)$ be one of the seven linearized identities produced by the previous step.
Each such $f$ produces seven multilinear identities in degree 7:
\[
f( x_1, \dots, x_i x_7, \dots, x_6 ) \quad (1 \le i \le 6),
\quad\quad\quad
f( x_1, \dots, x_6 ) x_7.
\]
We obtain 49 polynomial identities $g \equiv 0$ in degree 7, called the \emph{consequences of 
the known identities in degree 6}.

Step 4.
For each of the 15 partitions $\lambda$ of 7 with corresponding irreducible representation of dimension 
$d_\lambda$, we construct the $30 d_\lambda \times 11 d_\lambda$ matrix $S_\lambda$ in which 
the $(i,j)$ block of size $d_\lambda$ is the matrix in representation $\lambda$ for the element 
of the group algebra given by the term (if any) of symmetry $i$ in association type $j$.
We call $\mathrm{rank}(S_\lambda)$ the \emph{rank of the symmetries for partition $\lambda$}.

Step 5.
For each partition $\lambda$ of 7, we construct the $49 d_\lambda \times 11 d_\lambda$ matrix 
$C_\lambda$ in which the $(i,j)$ block is the matrix in representation $\lambda$ 
for the element of the group algebra given by the terms (if any) of consequence $i$ in association type $j$.
We call $\mathrm{rank}(C_\lambda)$ the \emph{rank of the consequences for partition $\lambda$}.

Step 6.
For each partition $\lambda$ of 7, we stack (combine vertically) $S_\lambda$ and $C_\lambda$ 
to obtain the $79 d_\lambda \times 11 d_\lambda$ matrix $SC_\lambda$.
We call $\mathrm{rank}(SC_\lambda)$ the \emph{rank of the old identities for partition $\lambda$},
where \emph{old identities} means identities in degree 7 which are consequences of known identities
of lower degree (the symmetries of the association types are the consequences of commutativity).

Step 7.
For each partition $\lambda$ of 7, we construct the full rank matrix $N_\lambda$ in row canonical form 
whose row space is the nullspace of the restriction of the expansion map $X$ to the isotypic component
of the $S_7$-module $\free(7)$ for representation $\lambda$.
We call $\mathrm{rank}(N_\lambda)$ the \emph{rank of all identities for partition $\lambda$}.
Hence $N_\lambda$ has size $\mathrm{rank}(N_\lambda) \times 11 d_\lambda$,
and the row space of $SC_\lambda$ is a subspace of the row space of $N_\lambda$.
The difference $\mathrm{new}(\lambda) = \mathrm{rank}(N_\lambda) - \mathrm{rank}(SC_\lambda)$ 
is always nonnegative; we call it the \emph{rank of the new identities for partition $\lambda$}.

We summarize these computations in Figure \ref{new7}, which completes the proof.
\end{proof}

\begin{figure}[ht]
$\begin{array}{
l|
r@{\;\;}r@{\;\;}r@{\;\;}r@{\;\;}r@{\;\;}
r@{\;\;}r@{\;\;}r@{\;\;}r@{\;\;}r@{\;\;}
r@{\;\;}r@{\;\;}r@{\;\;}r@{\;\;}r
}
\lambda & 7 & 61 & 52 & 51^2 & 43 & 421 & 41^3 & 3^21 & 32^2 & 321^1 & 31^4 & 2^21 & 2^21^3 & 21^5 & 1^7
\\ 
\midrule
\mathrm{rank}(S_\lambda)  & 
0 & 31 &  94 & 120 & 105 & 295 & 190 & 185 & 186 & 335 & 155 & 134 & 145 & 65 & 11  
\\
\mathrm{rank}(SC_\lambda) & 
7 & 40 & 100 & 122 & 106 & 296 & 190 & 185 & 186 & 335 & 155 & 134 & 145 & 65 & 11
\\ 
\mathrm{rank}(N_\lambda)  & 
7 & 47 & 106 & 126 & 111 & 301 & 191 & 188 & 187 & 336 & 155 & 134 & 145 & 65 & 11
\\ \midrule
\mathrm{new}(\lambda)     & 
0 &  7 &   6 &   4 &   5 &   5 &   1 &   3 &   1 &   1 &   0 &   0 &   0 &  0 &  0
\end{array}$
\caption{Multiplicities of irreducible representations in degree 7}
\label{new7}
\end{figure}



\end{document}